\documentclass[11pt]{article}

\usepackage{color}
\usepackage{amsfonts}
\usepackage{amsmath}
\usepackage{amssymb}
\usepackage{graphicx}
\usepackage{mathptmx}
\usepackage{subfigure}
\usepackage{float}
\usepackage{parskip}
\RequirePackage[colorlinks, hyperindex, plainpages=false]{hyperref}
\usepackage{amsmath}
\usepackage{dsfont}

\def\P{{\mathbb P}}
\def\E{{\mathbb E}}

\def\N{{\mathbb N}}
\def\R{{\mathbb R}}

\def\1{{\mathbf 1}}

\def\D {\mathcal{D}}
\def\L {\bar{L}}
\def\z {\mathbf{z}}

\newtheorem{theorem}{Theorem}[section]
\newtheorem{lemma}[theorem]{Lemma}
\newtheorem{corollary}[theorem]{Corollary}
\newtheorem{proposition}[theorem]{Proposition}
\newtheorem{theor}[theorem]{Theorem}

\newtheorem{remark}[theorem]{Remark}

\newenvironment{proof}[1][Proof.]{\textbf{#1} }{\hfill $\blacksquare$}
\def\beq{\begin{equation}}
\def\eeq{\end{equation}}
\newcommand{\bei}{\begin{itemize}}
\newcommand{\eei}{\end{itemize}}
\newcommand{\ben}{\begin{enumerate}}
\newcommand{\een}{\end{enumerate}}
\newcommand{\beqn}{\begin{eqnarray}}
\newcommand{\beqnn}{\begin{eqnarray*}}
\newcommand{\eeqn}{\end{eqnarray}}
\newcommand{\eeqnn}{\end{eqnarray*}}
\newcommand{\brm}{\begin{rmk}}
\newcommand{\erm}{\end{rmk}}

\begin{document}

\title{Functional law of large numbers and central limit theorem for Crump-Mode-Jagers branching processes}

\author{
Ibrahima~Dram\'{e}  \footnote{  \scriptsize{Université Cheikh Anta Diop de Dakar, FST, LMA, 16180 Dakar-Fann, S\'en\'egal. iboudrame87@gmail.com
}}
\and
Etienne~Pardoux \setcounter{footnote}{6}\footnote{ \scriptsize {Aix-Marseille Universit\'{e}, CNRS,  Centrale Marseille, I2M, UMR 7373, 13453 Marseille, France. etienne.pardoux@univ-amu.fr}}
}

\maketitle

\begin{abstract}
We establish a Law of Large Numbers and a Central Limit Theorem for a class of Crump Mode Jagers continuous time branching processes, where the birth rate is age dependent, and also random (different from one individual to the next), in the limit of a large number of ancestors. The only difficulty concerns the tightness in the Skorohod space $D$  for the central limit theorem. We exploit a criterion for the CLT in $D$ due to M. Hahn \cite{MH}. 
\end{abstract}

\vskip 3mm
\noindent{\textbf{Keywords}: } CMJ Branching processes; Poisson random measures; Law of large numbers; Central limit theorem; Tightness.
\vskip 3mm

\section{Introduction}
Crump-Mode-Jagers processes (CMJ-processes), as general continuous time and discrete state space branching process models with age-dependent reproduction mechanism, were introduced in \cite{CJ}, \cite{JA} with the study of biological populations as their main motivation. A CMJ-process is usually described as follows. It starts with a single individual at time 0 and this individual gives birth during its lifetime to a random number of offspring at a sequence of random times. Every child that is born evolves in the same way, independently of all other individuals in the population. In most applications, the lifetimes do not follow an exponential distribution, hence the process is not Markovian. 

Such processes are used as models in many applications. For instance, it is well--known that  the start of an epidemic is well approximated by a branching process, see \cite{BP}, and the same is true for the end of an epidemic, see
\cite{MDNP}, where the epidemic model is a general non Markovian model, thus leading to a CMJ process which is exactly of the type which will be considered in the present paper. CMJ branching processes have recently been used to model the spread of a disease (see \cite{BGMS}, \cite{OS}) or the allelic distribution of a population which reproduces according to a CMJ process, and is subject to mutations at a certain rate, see \cite{NCAL}, \cite{NCAL2}. Other examples of application appears in queuing theory, see e. g. \cite{GS}, \cite{LSZ}. 

Note that CMJ-processes are neither Markov processes (unless the lifetime distribution is exponential) nor semi-martingales, methods and tools developed for Galton-Watson processes usually can not be applied to study CMJ-processes. In particuar, the law of such a process cannot be easily computed, so that approximations of that law can be useful, whenever it is available.  This is our motivation for establishing  a functional law of large numbers and a functional central limit theorem, in the limit of the number of ancestors tending to infinity.   

Let us first describe the precise class of CMJ processes which will be considered in this paper. The law of the lifetime of each individual will be very general. The birth times of the offspring of each individual (after his/her birth) will be described as the jump times of the following counting process:
\[ A(t)=\int_0^t\int_0^\infty{\bf1}_{u\le b(s)}Q(ds,du),\]
where $b$ and $Q$ are independent, $b$ is an $\R_+$ valued process, $b(t)$ being the rate at which the considered individual of age $t$ gives birth to offspring, while $Q$ is a standard Poisson Random Measure on $\R_+^2$. Note that it is quite natural to assume that the birth rate depends upon the age, and vary from one individual to the next. Moreover, we will allow multiple births  at each birth time. 

As already explained, our goal is to approximate, via both the law of large numbers and the central limit theorem,
the law of our CMJ branching process, in case of a large number of ancestors. It is fair to say that the law of large numbers and the finite dimensional convergence in the functional CLT are easy consequences of very classical results. The only
difficulty in our work is to establish the functional convergence in the Skorohod space $D$ for the central limit theorem.  For that sake, our basic tool is the result of \cite{MH}, which gives a sufficient condition for the CLT in $D$. Our technique for verifying the conditions in the main theorem of \cite{MH} uses formulas for some moments of products of integrals of a deterministic function w.r.t. a compensated Poisson Random Measure (denoted below PRM).  

Note that functional law of large numbers and central limit theorems have been obtained for age-structured population processes, which are more general than branching process, see in particular \cite{Oels} and \cite{FJSwang}. However, having a birth rate which is not only age dependent but also random (different from one individual to the next) seems to be new. Of course, the fluctuations of the birth rate adds an additional term in the CLT. Note that the authors expect to generalize soon their work to take into account mean-field interaction.

The paper is organized as follows : in Section 2 we describe precisely the class of  Crump Mode Jagers continuous time branching processes to which our results apply. Section 3 is devoted to the statements and a good part of the proofs of our results, namely the Law of Large Numbers and the Central Limit Theorem, in the limit of a large number of ancestors. The most technical results needed for the proof of the Central Limit Theorem are established in Section 4, where we exploit the sufficient condition for the CLT in $D$ due to M. Hahn  \cite{MH}, and establish en estimate for the second moment of the product of two integrals with respect to a compensated Poisson random measure.

 We shall assume that all random variables in the paper are defined on the same probability space $( \Omega,\mathcal{F},\mathbb{P})$. We shall use the following notations: $\mathbb{Z}_+=\{0,1, 2,... \}$, $\mathbb{N}=\{1, 2,... \}$, $\mathbb{R}=(-\infty, \infty)$ and $\mathbb{R}_+=[0, \infty)$. For $x$ $\in$ $\mathbb{R}_+$, $[x]$ denotes the integer part of $x$. Let $L^2(\Omega)$ be the space of square integrable random variables, $C([0,\infty), \mathbb{R}_+)$ denote the space of $\mathbb{R}_+$-valued continuous functions defined on $[0,\infty)$, $D=D([0,\infty), \mathbb{R})$ and $D_+=D([0,\infty), \mathbb{R}_+)$ denote resp. the space of functions from $[0,\infty)$ into $\mathbb{R}$, resp. $\mathbb{R}_+$, which are right continuous and have left limits at any $t>0$ (as usual such a function is called c\`adl\`ag). We shall always equip the spaces $D$ and $D_+$ with the Skorohod topology, as described in Chapter 3 of \cite{Bill}. For $x,y\in \R$, we use the notations $x \wedge y = \min\{x,y\}$ and $x \vee y= \max\{x,y\}$.  In this paper a unique letter $C$ will denote a constant which may differ from line to line. 

\section{Preliminaries}
We first consider the evolution in time of a population started with a single ancestor. In other words, we shall be interested in the random function $Z(t)$ 
representing the number of individuals alive at time $t$ who are descendants of a single individual born at time $t=0$. Each individual appearing in the population lives for a random length of time $\eta$, and at random times during its life gives birth to offsprings that behave in a similar fashion, the behavior of each individual being independent of all others. Let $0< t^{(1)}< t^{(2)}< \cdots < t^{(n)}< \cdots$ be the birth times of the offsprings of the ancestor. Then the random function 
$A(t)= \sum_{j=1}^{\infty} \mathds{1}_{t^{(j)}\le t} $
is the number of birth events of offsprings of the ancestor up to time $t$. Clearly, we have 
\begin{equation}\label{eqZt1}  
Z(t)= \mathds{1}_{\eta >t} + \sum_{j=1}^{A(t)} \sum_{i=1}^{L_j}Z_{j,i}(t-t^{(j)}),
\end{equation} 
where $Z_{j,i}(t)  \stackrel{(d)}{=}  Z(t)$ and are independent and identically distributed (i.i.d) random variables, and $L_j$ denotes the number of twins born at the birth time  $t^{(j)}$. We assume that the birth events happen as follows.
\begin{equation*}
A(t)= \int_0^t \int_0^{b(s)} Q(ds,du),
\end{equation*}
where $Q(ds,du)$ is a Poisson random measures on $\R_+ \times \R_+$ with mean measure $dsdu$, and $b$ is an $\R_+$--valued random process independent of $Q$.  Let $\pi$ denote the common law of the random processes $\{Z_{j,i}(.), \ j, i \ge 1\}$. We can rewrite \eqref{eqZt1} as
\begin{equation}\label{eqZt2}   
Z(t)= \mathds{1}_{\eta >t} +  \int_0^t \int_0^{b(s)}\int_{\D} \sum_{i=1}^\ell z_i(t-s) \mathcal{M}(ds,du,d\z),
\end{equation} 
where $\mathcal{M}(ds,du,d\z)$ is a Poisson random measures on $\R_+ \times \R_+ \times \D$ with mean measure $\mu(ds,du,d\z)=dsdu\Pi(d\z)$, again independent of $b$, and such that $Q$ is the projection of $\mathcal{M}$ on the first two coordinates.
$\Pi$ is a probability distribution on $\D:=\N\times D^\infty$. It is the law of $(L, Z_1,Z_2,\ldots,)$, where $L, Z_1,Z_2, \ldots$ are mutually independent, the law of $L$ is that
of the number of twins at each birth event and $Z_1,Z_2, \ldots$ are i.i.d., all having the law $\pi$. We shall denote below by $\L$ the expectation of $L$.
Above $\z$ stands for $(\ell,z_1,z_2,\ldots)$.

The joint distribution of $b$ and $\eta$ is the same for each individual, but we do not require that $b$ and $\eta$ be independent. On the other hand, we assume that there is independence between the pair $(b, \eta)$ and $\mathcal{M}$.

Let us define $\bar{b}(t)=\E(b(t))$ the expectation of the random function $b$, $F(t)=\P(\eta\le t)$ the cumulative distribution function of $\eta$,  $F^c(t)=1-F(t)$ its survival function or reliability function and $M_1(t)= \E(Z(t))$ the expectation of the random process $Z$. It is easy to see that the function $M_1(t)$ satisfies the integral equation
\begin{equation}\label{eqMt1}   
M_1(t)= F^c(t)+  \int_0^t  \L M_1(t-s) \bar{b}(s) ds.
\end{equation} 
 Let us rewrite \eqref{eqZt2} in the following form, with $\bar{\mathcal{M}}=\mathcal{M}-\mu$,
\begin{equation}\label{eqZt3}   
Z(t)= \mathds{1}_{\eta >t} +\int_0^t \L M_1(t-s)  b(s) ds  + \int_0^t \int_0^{b(s)} \int_{\D} \sum_{i=1}^\ell z_i(t-s)  \bar{\mathcal{M}}(ds,du,d\z)\,.
\end{equation}

\section{Functional law of large numbers and central limit theorem}
Let $x>0$ be arbitrary, and let $N\geqslant1$ be an integer which will eventually go to infinity. Let $(Z^{N,x}(t))_{t\geq 0}$ denote the CMJ-process which describes the number of descendants alive at time $t$ of $[Nx]$ ancestors. We attribute to each individual in the population a mass equal to $1/N$ and we want to study the behaviour of the process $X^{N,x}(t)= N^{-1}Z^{N,x}(t)$. Since $Z^{N,x}$ satisfies the branching property, it can be written as $Z^{N,x}(t)= \sum_{k=1}^{[Nx]} Z_k(t)$, where for any $k\ge 1$, $Z_k(t)$ is a CMJ-process which describes the number of descendants alive at time $t$ of a single ancestor. Thus, we obtain 
\begin{equation} \label{eqXNxt1} 
X^{N,x}(t)= N^{-1} \sum_{k=1}^{[Nx]} Z_k(t). 
\end{equation} 
\subsection{Functional law of large numbers}
In this subsection, we only need to assume that the random function $b$ and $L$ satisfy the following assumption
\begin{equation*}
{(\bf H_1)}  :  \quad  \L<\infty,\ \text{ and for any} \ T>0, \  \sup_{0<t<T} \bar{b}(t)<\infty.
\end{equation*} 
It clearly follows from ${(\bf H_1)}$, \eqref{eqMt1} and Gronwall's Lemma: 
 \begin{align}\label{M1estim}
\sup_{0<t<T} M_1(t)<\infty,\quad\text{for all }T>0\,.
\end{align}

 Let us now establish our functional Law of Large Numbers.
 \begin{theor} Suppose that Assumption ${(\bf H_1)}$ is satisfied. Then, as $N \rightarrow \infty$, 
$(X^{N,x}(t), t\ge 0)$ converges to $(xM_1(t), t\ge 0)$ almost surely in $D_+$, where $M_1$ is the solution of the integral equation \eqref{eqMt1}. 
 \end{theor} 
\begin{proof}
From \eqref{eqXNxt1}, we deduce 
\begin{equation*}
X^{N,x}(t)-xM_1(t)= \frac{[Nx]}{N} \left( \frac{1}{[Nx]} \sum_{k=1}^{[Nx]} Z_k(t) - M_1(t) \right) + \left( \frac{[Nx]-Nx}{N}\right) M_1(t).
\end{equation*} 
From Thereom 1 in \cite{RR}, the first term on the right  converges to $0$ a. s. in $D$, while the convergence to $0$ of the second term locally uniformly in $t$ follows from \eqref{M1estim}.
 \end{proof}
\subsection{Functional central limit theorem}
In the rest of this section, we study the convergence in distribution of the process defined by
\begin{equation*}
\mathcal{R}^{N,x}(t)= \sqrt{N}\left(X^{N,x}(t)-xM_1(t) \right).
\end{equation*} 
From \eqref{eqMt1}, \eqref{eqZt3} and \eqref{eqXNxt1}, we can rewrite $\mathcal{R}^N$ in the form 
 \begin{equation}\label{eqRNxt1}
\mathcal{R}^{N,x}(t)= \sqrt{\frac{[Nx]}{N}}\left(\mathcal{U}_1^{N,x}(t)+\mathcal{U}_2^{N,x}(t)+\mathcal{U}_3^{N,x}(t) \right)+ \varepsilon^{N,x}(t),
\end{equation} 
with $\varepsilon^{N,x}(t)=\left(\frac{[Nx]-Nx}{\sqrt{N}}\right)\L\, M_1(t)$, and\footnote{The first term reflects the randomness of $\eta$, the second one the randomness of $b$, and the third one the randomness of $\mathcal{M}$.} 
 \begin{align*}
\mathcal{U}_1^{N,x}(t)= \sqrt{[Nx]} \left( \frac{1}{[Nx]}\sum_{k=1}^{[Nx]}  \mathds{1}_{\eta_k >t}-  F^c(t) \right)&,\
\mathcal{U}_2^{N,x}(t)=  \frac{1}{\sqrt{[Nx]}} \sum_{k=1}^{[Nx]}\int_0^t  \left( b_k(s) - \bar{b}(s)  \right)\L\, M_1(t-s)ds,
\\
\mathcal{U}_3^{N,x}(t)=  \frac{1}{\sqrt{[Nx]}} &\sum_{k=1}^{[Nx]} \int_0^t \int_0^{b_k(s)} \int_{\D} \sum_{i=1}^\ell z_i(t-s)  \bar{\mathcal{M}}_k(ds,du,d\z).
\end{align*}
where $(\eta_k)_{k\ge 1}$, $\{b_k(t), t\ge 0\}_{k\ge 1}$ and $(\mathcal{M}_k)_{k\ge 1}$ are respectively identically distributed independent copies of $\eta$, $\{b(t), t\ge 0\}$ and $\mathcal{M}$, defined in \eqref{eqZt2}. Moreover the pair $(b_k, \eta_k)$ and $\mathcal{M}_k$ are independent, for all $k\ge 1$.

In the next two results, we shall assume that $L$ and the random function $b$ satisfy the following hypothesis
\begin{equation*}
{(\bf H_2)}  :  \quad  \E[L^2]<\infty,\ \text{ and for any} \ T>0, \ \sup_{0<t<T} \E(b^2(t))<\infty .
\end{equation*} 
We shall need
 \begin{proposition}\label{M3tC}
Suppose that Assumption ${(\bf H_2)}$ is satisfied. For any $T>0$, there exists $C$ such that 
 \begin{align*}
\sup_{0<t<T} \E([Z(t)]^2)\le C.
\end{align*} 
  \end{proposition} 
\begin{proof}
From \eqref{eqZt3}, we have 
 \begin{align*}
\E([Z(t)]^2) \le 3 \left[ \P(\eta >t)+t\int_0^t M_1^2(t-s) \L^2\E(b^2(s)) ds  + \E\left(\int_0^t \int_0^{b(s)} \int_{\D} \sum_{i=1}^\ell z_i(t-s)  \bar{\mathcal{M}}(ds,du,d\z)\right)^2 \right],
\end{align*} 
However, from Proposition \ref{MomComp} below, we have that 
\[\E\left(\int_0^t \int_0^{b(s)} \int_{\D}\sum_{i=1}^\ell z_i(t-s)  \bar{\mathcal{M}}(ds,du,d\z)\right)^2= 
\int_0^t  \bar{b}(s)\{\L\E[(Z(t-s))^2]+M_1^2(t-s)[\E(L^2)-\L]\} ds.\]
Hence the result from Gronwall's Lemma. 
\end{proof}

The next statement follows readily from the usual central limit theorem and easy computations.
\begin{proposition}\label{fini-dim}
Suppose that Assumption ${(\bf H_2)}$ is satisfied. Then the final dimensional distributions of \\
$(\mathcal{U}^{N,x}_1,\mathcal{U}^{N,x}_2,\mathcal{U}^{N,x}_3)$ converge as $N\rightarrow +\infty$, to those of the centered Gaussian random vector\\ $(\mathcal{U}_1,\mathcal{U}_2,\mathcal{U}_3)$, whose covariances are given as follows
 \begin{align*}
\mathbb{C}ov(\mathcal{U}_1(t),\mathcal{U}_1(s))&=F^c(t \vee s)-F^c(t)F^c(s)\\
\mathbb{C}ov(\mathcal{U}_2(t),\mathcal{U}_2(s))&=\int_0^t \int_0^s \L^2 M_1(t-r)M_1(s-u) \mathbb{C}ov\left( b(r), b(u) \right)dudr\\
\mathbb{C}ov(\mathcal{U}_3(t),\mathcal{U}_3(s))&=\int_0^{t\wedge s}\left\{\L \bar{b}(r) \E[Z(t-r)Z(s-r)]+\E[L^2-L]M_1(t-r)M_1(s-r)\right\}dr\\
\mathbb{C}ov(\mathcal{U}_1(t),\mathcal{U}_3(s))&= \mathbb{C}ov(\mathcal{U}_2(t),\mathcal{U}_3(s))=0,\
\mathbb{C}ov(\mathcal{U}_1(t),\mathcal{U}_2(s))= \int_0^s \L M_1(s-r) \mathbb{C}ov\left( b(r) ,\mathds{1}_{\eta >t} \right)dr.
\end{align*}
\end{proposition}

For the rest of this subsection, we shall assume that the random function $b$ satisfies the following hypothesis.
\begin{equation*}
{(\bf H_3)}  :  \quad \E[L^4]<\infty\ \text{ and for any} \ T>0, \  \sup_{0<t<T} \E(b^4(t))<\infty.
\end{equation*} 
We shall need
 \begin{proposition}\label{M4tC}
Suppose that Assumption ${(\bf H_3)}$ is satisfied. For any $T>0$, there exists $C$ such that 
 \begin{align*}
\sup_{0<t<T} \E([Z(t)]^4)\le C.
\end{align*} 
  \end{proposition} 
\begin{proof}
From \eqref{eqZt3}, we have 
  \begin{align*}
\E([Z(t)]^4) \le 3^3 \left[ \P(\eta >t)+t^3 \L^4\int_0^t M_1^4(t-s) \E(b^4(s)) ds  + \E\left(\int_0^t \int_0^{b(s)} \int_{\D}\sum_{i=1}^\ell z_i(t-s)  \bar{\mathcal{M}}(ds,du,d\z)\right)^4 \right],
\end{align*} 
However, from Proposition \ref{MomComp} below, we have that 
\begin{align*}
\E\left(\int_0^t \int_0^{b(s)} \int_{\D}\sum_{i=1}^\ell z_i(t-s)  \bar{\mathcal{M}}(ds,du,d\z)\right)^4&= \int_0^t  \bar{b}(s)\E\left[\left|\sum_{i=1}^LZ_i(t-s)\right|^4\right] ds+ 3 \left( \int_0^t  \bar{b}(s)\E\left[\left|\sum_{i=1}^LZ_i(t-s)\right|^2\right] ds \right)^2\\
&\le  \int_0^t  \bar{b}(s)\E[L^4]\E[|Z(t-s)|^4]ds+3\left( \int_0^t  \bar{b}(s) \E[L^2]\E[|Z(t-s)|^2]ds \right)^2.
\end{align*}
The rest of the proof exploits Proposition \ref{M3tC} and Gronwall's Lemma. 
\end{proof}

Our assumption concerning the d.f. $F$ of the r.v. $\eta$ will be 
\begin{equation*}
{(\bf H_4)}  :  \quad \mbox{For any}\ T>0, \ \mbox{there exists} \ \alpha \in ]0,1] \ \mbox{and a constant} \ C>0 \ \mbox{such that} 
\end{equation*} 
\[ F(t)-F(s) \le C |t-s|^{\alpha}, \quad \forall 0\le s<t\le T.\]
We have the second main result of this section
 \begin{theor}\label{TCL} Suppose that Assumptions ${(\bf H_3)}$ and ${(\bf H_4)}$ are satisfied. Then, as $N \rightarrow \infty$, 
$(\mathcal{R}^{N,x}(t), t\ge 0)$ converges to $(\mathcal{R}^{x}(t), t\ge 0)$ in distribution in $D$, where $\mathcal{R}^{x}$ is given by
\begin{equation*}
 \mathcal{R}^{x}(t)=\sqrt{x} \left( \mathcal{U}_1(t)+\mathcal{U}_2(t)+\mathcal{U}_3(t) \right), \quad \mbox{with}
\end{equation*}
$(\mathcal{U}_1, \mathcal{U}_2, \mathcal{U}_3)$ is the Gaussian random element of $C([0,\infty))^3$ whose law is specified in Proposition \ref{fini-dim}. 
%
 \end{theor}
To prove the Theorem, we will need some preliminary results. The following is Theorem 14.3 in \cite{Bill}.
\begin{lemma}\label{lemtigU1N}
As $N \rightarrow \infty$, 
$(\mathcal{U}_1^{N,x}(t), t\ge 0)$ converges to $(\mathcal{U}_1(t), t\ge 0)$ in distribution in $D$.
\end{lemma}
We shall need the 
\begin{lemma}\label{lemtigU2N}
Suppose that Assumption ${(\bf H_2)}$ is satisfied. Then, as $N \rightarrow \infty$, $(\mathcal{U}_2^{N,x}(t), t\ge 0)$ converges to $(\mathcal{U}_2(t), t\ge 0)$ in distribution in $C([0,+\infty))$.
\end{lemma}
\begin{proof}
Let us rewrite $\mathcal{U}_2^{N,x}(t)$ in the form 
 \begin{align*}
\mathcal{U}_2^{N,x}(t)= \int_0^t \L\, V^{N,x}(s) M_1(t-s)ds,\quad \mbox{where} \quad V^{N,x}(t)&=  \frac{1}{\sqrt{[Nx]}} \sum_{k=1}^{[Nx]} \left(b_k(t) - \bar{b}(t)\right).
\end{align*}
From the theorem in \cite{KNSV}  (see also Lemma 1.6 of \cite{PRY}), we deduce that $(V^{N,x}(t), t\ge 0)$ converges to $(V(t), t\ge 0)$ in $L^2([0,\infty))$, where $V$ is the centered Gaussian random element of $C([0,\infty))$ specified by $\E(V(t)V(s))=\mathbb{C}ov(b(t)b(s))$, for $s,t>0$. Now, for $T>0$, let us define the linear map $\Phi$ from $L^2([0,T])$ into $C([0,T])$ by
 \begin{equation*}
 \Phi(f)=\int_0^{\bullet} f(s) \L\, M_1(\bullet- s)ds.
 \end{equation*} 
It remains to show that $\Phi$ is continuous, whose verification is left to the reader. 
 \end{proof}
 
 We need the following result, whose proof is postponed to the appendix. 
 \begin{lemma}\label{lemtigU3N}
The sequence $(\mathcal{U}_3^{N,x}, N\ge 1)$ is tight in $D$ and any limit of a subsequence belongs $a.s.$ to $C([0,\infty))$.
\end{lemma}


 We have
 \begin{proposition}\label{somdesUN}
As $N\rightarrow +\infty$, $ (\mathcal{U}_1^{N,x}(t)+ \mathcal{U}_2^{N,x}(t)+ \mathcal{U}_3^{N,x}(t), t\ge 0) \Longrightarrow (\mathcal{U}_1(t)+ \mathcal{U}_2(t)+ \mathcal{U}_3(t), t\ge 0)$ in $D$.
\end{proposition}
 \begin{proof}
 It follows from Lemmas \ref{lemtigU1N}, \ref{lemtigU2N} and \ref{lemtigU3N} that the sequence 
 $\{(\mathcal{U}_1^{N,x},\mathcal{U}_2^{N,x},\mathcal{U}_3^{N,x})\ N\ge1\}$ is tight in $D^3$, and we know that its unique limit is a.s. continuous. The result follows.
 \end{proof}

We can now turn to the

$\bf{Proof \ of \  Theorem}$  \ref{TCL} : Since thanks to \eqref{M1estim}, $\varepsilon^{N,x}(t)\to0$ locally uniformly in $t$, the result follows by combining \eqref{eqRNxt1} and Proposition \ref{somdesUN}.
$\hfill \blacksquare$

 \section{Appendix : Proof of Lemma \ref{lemtigU3N}}
 \subsection{Moment estimates}
\begin{proposition}\label{MomComp}
Let $(E, \mathcal{E})$ be a measurable space. Let $\mathbf{Q}$ be a Poisson random measure on $(E, \mathcal{E})$ with mean $\mu$, and $\bar{\mathbf{Q}}$ its associated 
compensated measure. If $f_i\in L^1(\mu) \cap L^4(\mu)$, 
\begin{align*}
\E[(\bar{\mathbf{Q}}(f))^2]=\mu(f^2) \quad \mbox{and} \quad \E[(\bar{\mathbf{Q}}(f))^4]&=\mu(f^4) + 3[\mu(f^2)]^2, 
\end{align*}
and if $f_i\in L^1(\mu) \cap L^4(\mu)$, $i=1,2$, 
 \begin{align*}
\E \left[(\bar{\mathbf{Q}}(f_1))^2\bar{\mathbf{Q}}(f_2)^2 \right]=\mu(f_1^2f_2^2)+\mu(f_1^2)\mu(f_2^2)+2[\mu(f_1f_2)]^2\,.
\end{align*}
\end{proposition}
Those formulas are easily deduced from the following formula for the characteristic function of $\bar{\mathbf{Q}}(f)$:
$\E[e^{i\bar{\mathbf{Q}}(f)}]=\exp(\mu(e^{if}-if-1))$.
\subsection{Proof of Lemma \ref{lemtigU3N}} 
%

Our Lemma \ref{lemtigU3N} will follow from Theorem 2 in \cite{MH}, if the process
\begin{align*}
 U(t):=\int_0^t\int_0^{b(r)}\int_\D \sum_{i=1}^\ell z_i(t-r) \bar{\mathcal{M}}(dr,du,d\z)
\end{align*}
satisfies the conditions of that Theorem. This section is devoted to the proof of the following Proposition.
\begin{proposition}\label{HanCondi}
We assume that the assumptions ${(\bf H_3)}$ and ${(\bf H_4)}$ hold and fix $T>0$. There exist a constant $C$ such that, 
 for all $0\le s<v<t\le T$, with $\beta=1+\alpha/2$,
\begin{itemize}
\item{(i)} $\E[|U(t)-U(s)|^2]\le C(t-s)$,
\item{(ii)} $\E[|U(t)-U(v)|^2|U(v)-U(s)|^2]\le C(t-s)^\beta$.
\end{itemize}
\end{proposition}
In the remaining of this section, $0\le s<v<t\le T$, where $T$ will be fixed. We define $G(t)=C(t+F(t))$, where the positive constant $C$ may vary from line to line.
\subsubsection{Proof of condition $(i)$}
From the above definition of $U$, we can rewrite \eqref{eqZt3} in the form
\begin{align}\label{expressZ}
Z(t)= \mathds{1}_{\eta >t} +\int_0^t \L\,M_1(t-s) b(s) ds  + U(t).
\end{align}
 Thus, let us first estimate, for $0\le s<t$,  $|\E[Z(t)-Z(s)]|$, $\E[(Z(t)-Z(s))^2]$ and $\E[(Z(t)-Z(s))^4]$. From the previous equation, we deduce that 
 \begin{align}\label{ecartZ}
 Z(t)-Z(s)= -{\bf1}_{s<\eta\le t}+ \int_s^t \L\, M_1(t-r)b(r)dr+\int_0^s \L\,[M_1(t-r)-M_1(s-r)]b(r)dr + U(t)-U(s),
  \end{align}
  where 
 \begin{align}\label{ecartU}
U(t)-U(s)=\int_s^t\int_0^{b(r)}\int_\D \sum_{i=1}^\ell z_i(t-r)\bar{\mathcal{M}}(dr,du,d\z)
 +\int_0^s\int_0^{b(r)}\int_\D \sum_{i=1}^\ell [z_i(t-r)-z_i(s-r)]\bar{\mathcal{M}}(dr,du,d\z) \, .
 \end{align}
We first have : 
  \begin{lemma}\label{ecartEZ}
For $0< s<t$, we have
 \begin{align*}
 |\E[Z(t)-Z(s)]|= |M_1(t)-M_1(s)|\le G(t)-G(s).
 \end{align*}
 \end{lemma}
\begin{proof}
 Taking the expectation in the identity \eqref{ecartZ}, we obtain
 \begin{align*}
 M_1(t)-M_1(s)&= F(s)-F(t)+\int_s^tM_1(t-r)\L\,\bar{b}(r)dr+\int_0^s[M_1(t-r)-M_1(s-r)]\L\,\bar{b}(r)dr,\\
 |M_1(t)-M_1(s)|&\le F(t)-F(s)+C(t-s)+C\int_0^s|M_1(t-r)-M_1(s-r)|dr\,
 \end{align*}
where we have used \eqref{M1estim}. The result now follows from Gronwall's Lemma applied to $\varphi(u)=M_1(t-s+u)-M_1(u)$, $0\le s\le u$.
    \end{proof}

In what follows, we will need the 
\begin{remark}\label{MajorG}
For $0<s<t$, it is easy to check that (recall that $s<t\le T$ and $T$ is fixed)
 \begin{align*}
(1)& \quad \int_0^s[G(t-r)-G(s-r)]dr \le \int_0^t G(r) dr-  \int_0^sG(r)dr  \le C(t-s)\\
(2) &\quad  G(t)-G(s)  \le C(t-s)^\alpha.
 \end{align*}
 \end{remark}
 Let us first establish a basic Lemma, which we will use several times.
  \begin{lemma}\label{basic}
  Let $n\in\N$, $\xi$ be an $\N$--valued r.v. and $\{X_i, i\ge1\}$ a sequence of identically distributed rv's
  having moment of order $n$,
  which is globally independent of $\xi$. Then
  \[ \E\left[\left(\sum_{i=1}^\xi X_i\right)^n\right]\le \E[\xi^n]\E[X_1^n]\,.\]
  \end{lemma}
  \begin{proof}
  We first note that $(\sum_{i=1}^\xi X_i)^n\le\xi^{n-1}\sum_{i=1}^\xi X_i^n$, then take the conditional expectation given $\xi$, 
  and finally the expectation.
  \end{proof}
 
We shall need 
  \begin{lemma}\label{MajorG24}
For $0< s<t$, we have
 \begin{align*}
 \E[|Z(t)-Z(s)|^2]\le G(t)-G(s)\quad \mbox{and} \quad \E[|Z(t)-Z(s)|^4]\le G(t)-G(s).
 \end{align*}
 \end{lemma}
\begin{proof} From \eqref{ecartZ}, \eqref{ecartU} and Lemma \ref{basic}, we have
  \begin{align*}
  \E[|Z(t)-Z(s)|^2]&\le 5[F(t)-F(s)]+5\L^2(t-s)\int_s^t\E[b(r)^2]dr  +5\int_s^t \E[L^2]\E[(Z(t-r))^2] \bar{b}(r)dr  \\
  &\quad + 5s\L^2\int_0^s(M_1(t-r)-M_1(s-r))^2\E[b^2(r)]dr +5\int_0^s \E[L^2]\E[|Z(t-r)-Z(s-r)|^2]\bar{b}(r)dr.
  \end{align*}
However, using Lemma \ref{ecartEZ} and $(1)$ of Remark \ref{MajorG}, we note that 
  \begin{align*}
  5s\L^2\int_0^s(M_1(t-r)-M_1(s-r))^2\E[b^2(r)]dr \le C\int_0^s|M_1(t-r)-M_1(s-r)|dr \le C(t-s).
  \end{align*}
  Combining the last two estimates, our assumptions, Proposition \ref{M3tC} and the definition of $G$, we deduce that 
  \begin{align*}
  \E[|Z(t)-Z(s)|^2]&\le  G(t)-G(s)+C\int_0^s \E[|Z(t-r)-Z(s-r)|^2]dr,
  \end{align*}
  Hence the first assertion follows from Gronwall's Lemma applied to $\varphi(u)=\E[|Z(t-s+u)-Z(u)|^2)$, $0\le u\le s$.
  The second assertion is easily deduced by a similar argument, exploiting the fourth moment estimate in Proposition \ref{MomComp}. 
 \end{proof}

Let us now verify condition $(i)$. Using \eqref{ecartU}, Lemma \ref{basic}, Lemma \ref{MajorG24} and $(1)$ of Remark \ref{MajorG}, , we deduce that 
\begin{align}
\E[|U(t)-U(s)|^2]& \le 2 \int_s^t \E[L^2]\E[(Z(t-r))^2] \bar{b}(r)dr
+ 2\int_0^s \E[L^2]\E[|Z(t-r)-Z(s-r)|^2]\bar{b}(r)dr \label{estimU2}\\
&\le C(t-s)+C\int_0^s[G(t-r)-G(s-r)]dr\nonumber\\
&\le C(t-s). \nonumber
\end{align}

\subsubsection{Proof of condition $(ii)$}
We now turn to verifying condition $(ii)$. 
Recalling \eqref{ecartU}, we have that
\begin{align*}
U(t)-U(\nu)=\bar{\mathcal{M}}(f),\ \text{with }
f(r,u,\z)&=\left\{{\bf1}_{[0,v]}(r)\sum_{i=1}^\ell[z_i(t-r)-z_i(v-r)]+{\bf1}_{[v,t]}\sum_{i=1}^\ell z_i(t-r)\right\}{\bf1}_{u\le b(r)},\\
U(\nu)-U(s)=\bar{\mathcal{M}}(g),\ \text{with }
g(r,u,\z)&=\left\{{\bf1}_{[0,s]}(r)\sum_{i=1}^\ell[z_i(v-r)-z_i(s-r)]+{\bf1}_{(s,v]}\sum_{i=1}^\ell z_i(v-r)\right\}{\bf1}_{u\le b(r)}\,.
\end{align*}
Recall that $b$ and $\mathcal{M}$ are independent. 
Denoting by $\E_b$ the conditional expectation, given $b$, we have the following consequence of Proposition \ref{MomComp} . 
\begin{corollary}\label{moncondi}
We have $\E_b \left[(U(t)-U(\nu))^2(U(\nu)-U(s))^2 \right]=\mu(f^2g^2)+\mu(f^2)\mu(g^2)+2[\mu(fg)]^2\,.$
\end{corollary}

We define $A_1=\E[\mu(f^2g^2)]$, $A_2=\E[\mu(f^2)\mu(g^2)]$ and $A_3=\E[[\mu(fg)]^2]$. Through a simple computation we obtain
\begin{align}
A_1&=\int_0^s \bar{b}(r)\E\left\{\Big|\sum_{i,j=1}^L[Z_i(v-r)-Z_i(s-r)][Z_j(t-r)-Z_j(v-r)]\Big|^2 \right\}dr\nonumber\\
&\quad+\int_s^v\bar{b}(r)\E\left\{\Big|\sum_{i,j=1}^L Z_i(v-r)[Z_j(t-r)-Z_j(v-r)]\Big|^2\right\}dr\label{1}\\
A_2&=\E\left\{\left(\int_0^v{b}(r)\E\Big\{\Big|\sum_{i=1}^L [Z_i(t-r)-Z_i(v-r)]\Big|^2\Big\}dr+\int_v^t{b}(r)\E\Big\{\Big|\sum_{i=1}^L Z_i(t-r)\Big|^2\Big\}dr\right)\right.\nonumber\\
&\quad\quad \times\left.\left(\int_0^s{b}(r) \E\Big\{\Big|\sum_{i=1}^L [Z_i(v-r)-Z_i(s-r)]\Big|^2\Big\}dr+\int_s^v{b}(r)\E\Big\{\Big|\sum_{i=1}^L Z_i(v-r)\Big|^2\Big\}dr\right)\right\}\label{2}\\
A_3&= \E\left[\left(\int_0^sb(r)\E\Big\{\sum_{i,j=1}^L[Z_i(v-r)-Z_i(s-r)][Z_j(t-r)-Z_j(v-r)]\Big\}dr\right.\right.\nonumber\\
&\quad\quad\left.\left. +\int_s^vb(r)\E\Big\{\sum_{i,j=1}^L Z_i(v-r)[Z_j(t-r)-Z_j(v-r)]\Big\}dr\right)^2\right]\label{4}.
\end{align}
We have 
\begin{proposition}\label{proLONG}
For $0\le s< v<t$, we have 
 \begin{align}\label{estimA}
\sum_{i=1}^3 A_i \le C(t-s)^{1+\frac{\alpha}{2}}+ C \int_0^s \E\left\{[Z(v-r)-Z(s-r)]^2[Z(t-r)-Z(v-r)]^2\right\}dr.
\end{align}
\end{proposition}
\begin{proof}
The proof is organized as follows. Step $1$ establishes that 
\begin{equation}\label{A1}
 A_1 \le C (t- s)^{1+\frac{\alpha}{2}}+ C \int_0^s \E\left\{[Z(v-r)-Z(s-r)]^2[Z(t-r)-Z(v-r)]^2\right\}dr. 
 \end{equation}
Steps $2$ shows that $A_2+A_3 \le C(t-s)^{2}.$

$Step\ 1$. We first estimate the last term in \eqref{1}. 
In view of Proposition \ref{M4tC}, Lemma \ref{MajorG24} and Remark \ref{MajorG} (2), 
\begin{align*} 
\E\left\{\Big| Z_i(v-r)[Z_j(t-r)-Z_j(v-r)]\Big|^2\right\} &\le \left[ \E\big[Z(v-r)\big]^4 \right]^{1/2} (t-s)^{\alpha/2} \le C(t-s)^{\alpha/2},
\end{align*} 
from with we easily deduce that the last term in \eqref{1} is bounded by $C(t-s)^{1+\alpha/2}$.

The first term on the right of  \eqref{1} is bounded by the second term on the right of \eqref{A1} plus $C$ times
\begin{align*}
\int_0^s\E\{|Z(\nu-r)-Z(s-r)|^2\}\E\{|Z(t-r)-Z(\nu-r)|^2\}dr\le C(t-\nu)^{\alpha}\int_0^s(\nu-s)\le C (t-s)^{1+\alpha},
\end{align*}
where we have used twice Lemma \ref{MajorG24}, and both items of Remark \ref{MajorG}. \eqref{A1} is established.

$Step\ 2$. We now note that this step of the proof, we estimate quantities of the form 
\[ \left|\E\left[ \int_x^y\int_z^u b(r)\varphi(r)b(s)\psi(s) drds\right]\right|=\left|\int_x^y\int_z^u \E[b(r)b(s)]\varphi(r)\varphi(s)drds\right|\le C\int_x^y\int_z^u|\varphi(r)\varphi(s)|drds,\]
since $\varphi$ and $\psi$ are deterministic  functions, and $b$ satisfies assumption {\bf (H3)}.

In view of this indication, the inequality
$A_2 \le C(t-s)^{2}$
follows easily from Lemma \ref{basic}, Lemma \ref{MajorG24} and $(1)$ of Remark \ref{MajorG}.
 Essentially the same arguments yield $A_3 \le C(t-s)^{2}$.
\end{proof}

It remains to deduce Condition (ii) from Proposition \ref{proLONG}. 
For that sake, we need to estimate\\ $\E\left\{Z(t)-Z(\nu)|^2|Z(\nu)-Z(s)|^2\right\}$. Using the notation
\[ B(\alpha,\beta):=\left(\int_\alpha^\beta \bar{L}M_1(\beta-r)b(r)dr\right)^2+\left(\int_0^\alpha \bar{L}[M_1(\beta-r)-M_1(\alpha-r)]b(r)dr\right)^2,\]
we deduce from \eqref{ecartZ} that
\begin{align*}
|Z(t)-Z(\nu)|^2&\le{\bf1}_{\nu<\eta\le t}+B(\nu,t)+|U(t)-U(\nu)|^2,\\
|Z(\nu)-Z(s)|^2&\le {\bf1}_{s<\eta\le \nu}+B(s,\nu)+|U(\nu)-U(s)|^2\,.
\end{align*}
We first note that ${\bf1}_{\nu<\eta\le t}{\bf1}_{s<\eta\le \nu}=0$. Moreover, from Lemma \ref{ecartEZ} and (1) of Remark \ref{MajorG},
\begin{align*}
\E\left\{{\bf1}_{\nu<\eta\le t}B(s,\nu)+{\bf1}_{s<\eta\le \nu}B(\nu,t)\right\}\le \E\{B(s,\nu)+B(s,\nu)\}\le C(t-s)^2\,.
\end{align*}
It follows readily from Cauchy-Schwartz and ${(\bf H_4)}$ that
\[ \E\left({\bf1}_{\nu<\eta\le t}b(r)\right)\le C(t-s)^{\alpha/2},\quad \E\left({\bf1}_{\nu<\eta\le t}b(r)b(r')\right)\le C(t-s)^{\alpha/2}\,.\]
Using those inequalities and \eqref{estimU2} with $\E$ replaced by $\E_b$, we easily deduce that
\begin{align*}
\E\left({\bf1}_{\nu<\eta\le t}|U(\nu)-U(s)|^2+{\bf1}_{s<\eta\le \nu}|U(t)-U(\nu)|^2\right)\le C(t-s)^{1+\alpha/2}\,.
\end{align*}
Next it is not hard to show that
\[ \E\left(B(\nu,t)|U(\nu)-U(s)|^2+B(s,\nu)|U(t)-U(\nu)|^2\right)\le C(t-s)^2\,.\]
We conclude from those last estimates that
\begin{align*}
\E\left(|Z(t)-Z(\nu)|^2|Z(\nu)-Z(s)|^2\right)\le C(t-s)^{1+\alpha/2}+\E\left(|U(t)-U(\nu)|^2|U(\nu)-U(s)|^2\right)\,.
\end{align*}
Combining this with \eqref{estimA}, we deduce that
\begin{align*}
\E\left(|Z(t)-Z(\nu)|^2|Z(\nu)-Z(s)|^2\right)\le C(t-s)^{1+\alpha/2}+C \int_0^s \E\left\{[Z(v-r)-Z(s-r)]^2[Z(t-r)-Z(v-r)]^2\right\}dr\,.
\end{align*}
Replacing above $t$ by $t-s+r$, $\nu$ by $\nu+s-r$, and $s$ by $r$, and defining\\
$\psi(r)= \E\left(|Z(t-s+r)-Z(\nu-s+r)|^2|Z(\nu-s+r)-Z(r)|^2\right)$, we deduce from the above that
\[ \psi(r)\le C(t-s)^{1+\alpha/2}+C\int_0^r\psi(u)du,\quad \forall\, 0\le r\le s\,.\]
Hence from Gronwall's Lemma, $\psi(r)\le C(t-s)^{1+\alpha/2}$, for all $0\le r\le s$. This combined with \eqref{estimA} yields
(ii).
\frenchspacing
\bibliographystyle{plain}

\end{document}